\documentclass[twoside,a4paper,12pt,centertags]{amsart}
\usepackage{amsmath,amssymb,verbatim,vmargin}
\usepackage[bookmarks=true]{hyperref}   % Hyperref in DVI and PDF
\usepackage{color}

\theoremstyle{plain}
\newtheorem{thm}{Theorem}[section]
\newtheorem{lem}[thm]{Lemma}
\newtheorem{cor}[thm]{Corollary}
\newtheorem{prop}[thm]{Proposition}

\theoremstyle{definition}
\newtheorem{defn}[thm]{Definition}
\newtheorem{rem}[thm]{Remark}

\mathchardef\semic="303B
\newcommand{\R}{{\mathbf R}}
\newcommand{\C}{{\mathbf C}}
\newcommand{\Z}{{\mathbf Z}}

\newcommand{\mX}{{\mathcal X}}
\newcommand{\mY}{{\mathcal Y}}

\newcommand{\mD}{{\mathcal D}}

\newcommand{\sett}[2]{ \{ #1 \, \semic \, #2 \} }
\newcommand{\bigsett}[2]{ \Big\{ #1 \semic #2 \Big\} }
\newcommand{\scl}[2]{\langle #1,#2 \rangle}

\newcommand{\dist}{\text{{\rm dist}}\,}
\newcommand{\diam}{\text{{\rm diam}}\,}

\newcommand{\barint}{\mbox{$ave \int$}}
\newcommand{\divv}{{\text{{\rm div}}}}

\newcommand{\loc}{\text{{\rm loc}}}

\def\barint_#1{\mathchoice
            {\mathop{\vrule width 6pt
height 3 pt depth -2.5pt
                    \kern -8.8pt
\intop}\nolimits_{#1}}%
            {\mathop{\vrule width 5pt height
3 pt depth -2.6pt
                    \kern -6.5pt
\intop}\nolimits_{#1}}%
            {\mathop{\vrule width 5pt height
3 pt depth -2.6pt
                    \kern -6pt
\intop}\nolimits_{#1}}%
            {\mathop{\vrule width 5pt height
3 pt depth -2.6pt
          \kern -6pt \intop}\nolimits_{#1}}}

\usepackage{color}

\definecolor{gr}{rgb}   {0.,   0.8,   0. } 
\definecolor{bl}{rgb}   {0.,   0.5,   1. } 
\definecolor{mg}{rgb}   {0.7,  0.,    0.7}

\makeatletter
\@namedef{subjclassname@2010}{%
  \textup{2010} Mathematics Subject Classification}
\makeatother

\begin{document}

\title[On the Carleson duality]
{On the Carleson duality}
\author[Tuomas Hyt\"onen]{Tuomas Hyt\"onen} 
\author[Andreas Ros\'en]{Andreas Ros\'en$\,^1$}
\thanks{$^1\,$Formerly Andreas Axelsson}
\address{Tuomas Hyt\"onen, Institutionen f\"or matematik och statistik, PB 68 (Gustaf H\"all\-str\"oms gata 2b), FI-00014 Helsingfors universitet, Finland} %{Department of Mathematics and Statistics, P.O. Box 68, FI-00014 University of Helsinki, Finland}
\email{tuomas.hytonen@helsinki.fi}
\address{Andreas Ros\'en, Matematiska institutionen, Link\"opings universitet, 581 83 Lin\-k\"oping, Sweden}
\email{andreas.rosen@liu.se}

\begin{abstract}
As a tool for solving  the Neumann problem for divergence form equations, Kenig and Pipher introduced the space $\mX$ of  functions on the half space, such that the non-tangential maximal function of their $L_2$-Whitney averages belongs to $L_2$ on the boundary. In this paper, answering questions which arose from recent studies of boundary value problems by Auscher and the second author, we find the pre-dual of $\mX$, and characterize the pointwise multipliers from $\mX$ to $L_2$ on the half space as the well-known Carleson-type space of functions introduced by Dahlberg. We also extend these results to $L_p$ generalizations of the space $\mX$.
Our results elaborate on the well-known duality between Carleson measures and non-tangential maximal functions.
\end{abstract}


% 42B25 (1980-now) Maximal functions, Littlewood-Paley theory
% 42B35 (2000-now) Function spaces arising in harmonic analysis 
% 42B37 (2010-now) Harmonic analysis and PDE


\maketitle

\section{Introduction}

A fundamental estimate in harmonic analysis is Carleson's inequality for Carleson measures. See \cite[Thm~2]{Carleson:58} and \cite[Thm~1]{Carleson:62} for the original formulations and applications in the theory of interpolating analytic functions, or for example Stein~\cite[Sec.~II 2.2]{stein:harm} and Coifman, Meyer and Stein~\cite{CMS} for more recent accounts in the framework of real-variable harmonic analysis. This inequality states that for a function $f(t,x)$ and a measure $d\mu(t,x)$ in the upper half space $\R^{1+n}_+:= \sett{(t,x)}{t>0, x\in \R^n}$, one has the estimate
$$
  \iint_{\R^{1+n}_+} |f(t,x)| d\mu(t,x) \lesssim \sup_Q (\mu(\widehat Q)/|Q|)\, \int_{\R^n} N_*f(y) dy,
$$
where the supremum is over all cubes $Q$ in $\R^n$ and $\widehat Q:= (0, \ell(Q))\times Q$ is the 
{\em Carleson box}, $\ell(Q)$ being the sidelength of $Q$.
Furthermore $N_*$ denotes the {\em non-tangential maximal function}
$$
  (N_*f)(y):= \sup_{\sett{(t,x)}{|x-y|\le at}} |f(t,x)|, \qquad y\in\R^n,
$$
where $a>0$ is a fixed constant determining the aperture of the cone.
The exact value of $a$ is less important, since for any $a_1,a_2>0$
the corresponding non-tangential maximal functions $N_* f$ are comparable in $L_p(\R^n)$ norm for any $1\le p\le \infty$.
See Fefferman and Stein~\cite[Lem. 1]{FS}.

Carleson's inequality has numerous applications. Motivating for this paper is its applications to 
boundary value problems for elliptic partial differential equations. 
A recent application concerns boundary value problems for divergence form equations
$\divv_{t,x} A(t,x) \nabla_{t,x}u(t,x)=0$, with non-smooth coefficients
$
  A\in L_\infty(\R^{1+n}_+; \C^{(1+n)\times(1+n)})
$
with uniformly positive real part.
To solve the Neumann problem with $L_2(\R^n)$ boundary data,
Kenig and Pipher~\cite{KP} introduced (a space equivalent to) the function space 
$\mX$ consisting of functions $f(t,x)$, thought of as gradients of solutions $u(t,x)$, with $N_*(W_2 f)\in L_2(\R^n)$, where
$$
  (W_q f)(t,x):= {|W(t,x)|}^{-1/q} \|f\|_{L_q(W(t,x))}, \qquad (t,x)\in \R^{1+n}_+,
$$
is the $L_q$ {\em Whitney averaged function}, with 
$$
W(t,x):= \sett{(x,y)\in \R^{1+n}_+}{|y-x|<c_1t, c_0^{-1}<s/t<c_0}
$$ 
being the {\em Whitney region} around $(t,x)$.
(Again, the precise value of the fixed constants $c_0>1$ and $c_1>0$ is less important.) 
The reason for replacing $f$ by the Whitney average $W_2f$ is that, unlike the potential $u(t,x)$,
the gradient $f(t,x)=\nabla_{t,x}u(t,x)$ does not have classical interior pointwise DeGiorgi--Nash--Moser bounds.

In the recent works of one of the authors with P. Auscher~\cite{AA1, AA2}, the function space $\mX$
above is fundamental. In these papers, new methods are developed to solve the Neumann (as well as the Dirichlet) problems for systems
of divergence form equations, which rely on solving certain operator-valued singular integral equations 
in this functions space $\mX$. Two questions arose, which motived this paper.
\begin{itemize}
\item Which functions $g(t,x)$ are bounded multipliers 
$$
  \mX\to L_2(\R^{1+n}_+;dtdx): f(t,x) \mapsto g(t,x)f(t,x)\, ?
$$
It was shown \cite[lem. 5.5]{AA1}, using Carleson's inequality, that $g$ is a multiplier if
the modified Carleson norm
\begin{equation}    \label{eq:modCarl}
  \sup_Q \left( \frac 1{|Q|}\iint_{\widehat Q} W_\infty g(t,x)^2 dtdx\right)^{1/2}
\end{equation}
is finite. We show in this paper (Theorem~\ref{thm:nondyadicduality}) 
that this modified Carleson norm in fact is equivalent to the
multiplier norm
$$
\|g\|_{\mX\to L_2(\R^{1+n}_+; dtdx)} =\sup_{f\ne 0}(\|gf\|_{L_2(\R^{1+n}_+; dtdx)}/\|f\|_\mX).
$$
The modified Carleson norm \eqref{eq:modCarl} has been known for some time to be fundamental in the perturbation theory for divergence form equations. It was introduced already by Dahlberg~\cite{D2}.
See also Fefferman, Kenig and Pipher~\cite{FKP} and Kenig and Pipher~\cite{KP, KP2}.
\item
What is the dual, or predual, space of $\mX$? 
We show in this paper (Theorem~\ref{thm:nondyadicdualspaces}) 
that $\mX$ is the dual space of the space of 
functions $g(t,x)$ such that
$$
  \int_{\R^n} \left( \sup_{Q\ni z}\iint_{\widehat Q} W_2 g(t,x) dtdx  \right)^2 dz <\infty.
$$
(We here identify a function $f\in \mX$ with the functional $g\mapsto \iint_{\R^{1+n}_+}fg dtdx$.)
Theorem~\ref{thm:nondyadicdualspaces} also shows that the space $\mX$ is not reflexive.
The interest in understanding duality for the space $\mX$ comes from the dual relation between 
the Dirichlet problem with $L_2(\R^n)$ data and the Dirichlet problem with Sobolev $H^1(\R^n)$ data.
See \cite[Thm. 5.4]{KP} and \cite[Thm. 1.4]{AA2}.
\end{itemize}
 
Beyond these two results, we prove more general $L_p$ results for the Carleson duality.
On one hand, we consider not only $W_\infty g$ and $W_2 g$, but more general $L_q$ Whitney averages.
On the other hand, we measure the non-tangential maximal function and the Carleson functional in $L_p$ norms.
For example, this may have useful applications to boundary value problems with $L_p$ data.

In Section~\ref{sec:dyadic}, we first prove the corresponding results for a discrete vector-valued model of the Carleson duality.
 Then in Section~\ref{sec:nondyadic}, we prove equivalence between dyadic and non-dyadic norms, 
 which yields the non-dyadic results.

The spaces we consider here are closely related to the tent spaces introduced by Coifman, Meyer and Stein \cite{CMS}, and in fact reduce to them for certain choices of the parameters. However, as a whole, the scale of spaces that we consider is new. Since the precise connection to tent spaces is somewhat technical, we postpone a more detailed commentary until Remark \ref{rem:tentspaces} below.

{\bf Acknowledgments.}

This work was done during a visit by the first author to Link\"oping university in connection with a workshop on ``harmonic analysis and elliptic PDEs'', organised by the second author and funded through the Tage Erlander prize 2009, the Swedish Research Council and Nordforsk.
The first author was supported by the Academy of Finland, grants 130166, 133264 and 218148.

\section{A discrete vector-valued model}    \label{sec:dyadic}

In this section we study a dyadic model of the problem. We use the following notation.
Let $\mD= \bigcup_{j\in \Z}\mD_j$ denote the dyadic cubes in $\R^n$, where
$$
  \mD_j:= \sett{2^{-j}(0,1)^n+2^{-j}k}{k\in \Z^n}.
$$
Let $W_Q:= (\ell(Q)/2, \ell(Q))\times Q$ denote the dyadic Whitney region, being in one-to-one
correspondence with $Q\in \mD$.
Note that unlike their non-dyadic counterparts $W(t,x)$, the regions $W_Q$ form a disjoint partition 
of $\R^{1+n}_+$ (modulo zero-sets).
Define the dyadic Hardy--Littlewood maximal function
$$
  M_\mD h(x):= \sup_{Q: x\in Q\in \mD} \frac 1{|Q|}\int_Q h(y) dy, \qquad x\in\R^n,
$$
for $h\in L_1^{\text{loc}}(\R^n)$. Recall that $M_\mD$ is bounded on $L_p(\R^n)$, $1<p\le \infty$.

Our discrete vector-valued setup is as follows.
We assume that to each $Q\in \mD$, there are two associated Banach spaces $\mX_Q$ and $\mY_Q$.
For a sequence $f=(f_Q)_{Q\in\mD}$, where $f_Q\in \mX_Q$, we define its non-tangential maximal function
$$
  (N_\mX f)(x):= \sup_{Q: x\in Q\in\mD} \|f_Q\|_{\mX_Q}, \qquad x\in \R^n.
$$
For fixed $1\le p<\infty$, let $\mX_p$ denote the space of all sequences $f$ such that 
$\|f\|_{\mX_p}:= \|N_\mX f\|_{L_p(\R^n)}<\infty$.
For a sequence $g=(g_Q)_{Q\in\mD}$, where $g_Q\in \mY_Q$, we define the Carleson functional
$$
  (C_\mY g)(x):= \sup_{Q: x\in Q\in\mD} \frac 1{|Q|}\sum_{R\subset Q, R\in\mD} \|g_R\|_{\mY_R},\qquad x\in \R^n.
$$
For fixed $1< p'\le \infty$, let $\mY_{p'}$ denote the space of all sequences $g$ such that 
$\|g\|_{\mY_{p'}}:= \|C_\mY g\|_{L_{p'}(\R^n)}<\infty$. 
Note that the case $p'=1$ is not interesting, since $g=0$ necessarily if $\|C_\mY g\|_{L_1(\R^n)}<\infty$.

We assume that for each $Q\in\mD$ there is a duality $\scl{\mX_Q}{\mY_Q}$ as below, with constants $C$ uniformly bounded with respect to $Q$.

\begin{defn}    \label{defn:duality}
Let $\mX$ and $\mY$ be two Banach spaces. By a {\em duality} $\scl{\mX}{\mY}$, we mean
a bilinear map $\mX\times \mY\ni (f,g)\mapsto \scl fg\in\R$ and a constant $0<C<\infty$
such that
\begin{gather*}
  |\scl fg|\le C \|f\|_\mX\|g\|_\mY, \qquad f\in \mX, \, g\in \mY, \\
  \|f\|_\mX \le C \sup_{\|g\|_\mY=1}\scl fg, \qquad f\in \mX, \\
  \|g\|_\mY \le C \sup_{\|f\|_\mX=1}\scl fg, \qquad g\in \mY.
\end{gather*}
\end{defn}

We prove the following duality result.

\begin{thm}    \label{thm:dyadic}
Let $(\mX_Q)_{Q\in\mD}$ and $(\mY_Q)_{Q\in\mD}$ be pairwise dual Banach spaces as above, 
and let $1/p+1/p'=1$, $1\le p<\infty$.
Then there is a constant $0<C<\infty$ such that 
\begin{gather*}
  \sum_{Q\in\mD} |\scl{f_Q}{g_Q}|\le C \|N_\mX f\|_{L_p(\R^n)}\|C_\mY g\|_{L_{p'}(\R^n)}, \qquad f_Q\in \mX_Q, \, g_Q\in \mY_Q, \\
  \|N_\mX f\|_{L_p(\R^n)} \le C \sup_{\|C_\mY g\|_{L_{p'}(\R^n)}=1} \sum_{Q\in\mD} \scl{f_Q}{g_Q}, \qquad f_Q\in \mX_Q, \\
  \|C_\mY g\|_{L_{p'}(\R^n)} \le C \sup_{\|N_\mX f\|_{L_p(\R^n)}=1} \sum_{Q\in\mD} \scl{f_Q}{g_Q}, \qquad g_Q\in \mY_Q.
\end{gather*}
\end{thm}

The application we have in mind is the following.
For functions $f(t,x)$ in $\R^{1+n}_+$, let $f_Q:= f|_{W_Q}\in L_q(W_Q)=: \mX_Q$, 
where the Banach space has norm
$\|f\|_{\mX_Q}:= |W_Q|^{-1/q}\|f_Q\|_{L_q(W_Q)}$ so that
$$
  N_{L_q}f= \sup_{Q\ni x, Q\in\mD} |W_Q|^{-1/q} \|f\|_{L_q(W_Q)}.
$$
For functions $g(t,x)$ in $\R^{1+n}_+$, let $g_Q:= g|_{W_Q}\in L_{\tilde q}(W_Q)=: \mY_Q$, 
where the Banach space has norm
$\|g\|_{\mY_Q}:= |W_Q|^{1-1/\tilde q}\|g_Q\|_{L_{\tilde q}(W_Q)}$ so that
$$
  C_{L_{\tilde q}}f= \sup_{Q\ni x, Q\in\mD}\frac 1{|Q|} 
  \sum_{R\subset Q, R\in\mD} |W_R|^{1-1/\tilde q} \|g\|_{L_{\tilde q}(W_R)}.
$$
We generalize slightly the Carleson functional and define
$$
  C^r_{L_{\tilde q}}f(x)=  \sup_{Q: x\in Q\in\mD}\left( \frac 1{|Q|}
  \sum_{R\subset Q, R\in\mD}|W_R| \Big(|W_R|^{-1/\tilde q} \|g\|_{L_{\tilde q}(W_R)}\Big)^r \right)^{1/r},
$$
for $x\in \R^n$ and $1\le r<\infty$.

\begin{cor}   \label{cor:dyadicwithr}
Let $1/p+1/\tilde p=1/q+1/\tilde q=1/r$, with $r\le p<\infty$, $r\le q\le \infty$, $1\le r<\infty$.
Then there is a constant $0<C<\infty$ such that 
\begin{gather*}
  \|fg\|_{L_r(\R^{1+n}_+)} \le C \|N_{L_q} f\|_{L_p(\R^n)} \|C^r_{L_{\tilde q}} g\|_{L_{\tilde p}(\R^n)}, \\
 \|N_{L_q} f\|_{L_p(\R^n)} \le C \sup_{\|C^r_{L_{\tilde q}} g\|_{L_{\tilde p}(\R^n)}=1} \|fg\|_{L_r(\R^{1+n}_+)}, \\
 \|C^r_{L_{\tilde q}} g\|_{L_{\tilde p}(\R^n)} \le C \sup_{ \|N_{L_q} f\|_{L_p(\R^n)} =1}  \|fg\|_{L_r(\R^{1+n}_+)}.
\end{gather*}
\end{cor}

Note that the case $p=q=r=2$ solves a dyadic version of the multiplier question for the space $\mX$ from the introduction. In this case $\tilde p= \tilde q=\infty$.
Note also that the case $p=q=2, r=1$, together with Theorem~\ref{thm:fulldyadicduality} below, solves a dyadic version of the dual space question for the space $\mX$ from the introduction. In this case $\tilde p= \tilde q=2$.

\begin{proof} 
  Replacing $|f|^r$, $|g|^r$ by $f, g$, we see that it suffices to consider the case $r=1$.
  In this case, the result follows from Theorem~\ref{thm:dyadic}.
\end{proof}

\begin{proof}[Proof of Theorem~\ref{thm:dyadic}]
(i)
For completeness, we start with the well-known proof of the $\sum_Q|\scl {f_Q}{g_Q}|$ estimate.
It suffices to estimate $\sum_Q \|f_Q\| \|g_Q\|$. 
Note that 
$$
\sum_{R\subset Q}\|g_R\|\le |Q| \inf_{x\in Q} C_\mY g(x)\le \int_Q C_\mY g,
$$
for any $Q\in \mD$.
Select, for given $k\in\Z$, the maximal dyadic cubes $\mD^k\subset \mD$ such that $\|f_Q\|>2^k$.
Then $\bigcup_{Q\in\mD^k} Q= \sett{x\in\R^n}{N_\mX f(x)>2^k}$, and the cubes in $\mD^k$ are disjoint.
We get
$$
  \sum_{Q: \|f_Q\|>2^k} \|g_Q\|\le \sum_{Q\in\mD^k} \sum_{R\subset Q} \|g_R\|\le \sum_{Q\in \mD^k}\int_Q C_\mY g= \int_{x: N_\mX f(x)>2^k} C_\mY g,
$$
and hence
\begin{multline*}
  \sum_{Q\in \mD} \|f_Q\|\|g_Q\|\approx \sum_{Q\in \mD}\sum_{k: 2^k<\|f_Q\|} 2^k \|g_Q\|
  = \sum_{k\in\Z} 2^k \sum_{Q: \|f_Q\|>2^k} \|g_Q\| \\
  \le  \sum_{k\in\Z} 2^k  \int_{x: N_\mX f(x)>2^k} C_\mY g
  = \int_{\R^n}\sum_{k: 2^k< N_\mX f(x)} 2^k C_\mY g \\
  \approx \int_{\R^n} N_\mX f C_\mY g \le \|N_\mX f\|_p \|C_\mY g\|_{p'}.
\end{multline*}

(ii)
Next we prove the estimate of $\|C_\mY g\|_{p'}$.
Consider first the case $p'=\infty$.
Pick $Q\in\mD$ such that $\tfrac 1{|Q|} \sum_{R\subset Q}\|g_R\| \ge \tfrac 12 \|C_\mY g\|_{\infty}$.
Then construct $f=(f_R)_{R\in\mD}$, choosing $f_R\in \mX_R$ such that $\|f_R\|=1/|Q|$, $\|g_R\|/|Q|\approx \scl{f_R}{g_R}$ if $R\subset Q$, and $f_R:= 0$ if $R\not\subset Q$.
It follows that $\|C_\mY g\|_{\infty}\approx \sum_R \scl{f_R}{g_R}$ and $\|N_\mX f\|_1=1$ since $N_\mX f=1/|Q|$ on $Q$ and $N_\mX f=0$ off $Q$.

Next consider the case $1<p'<\infty$.
Select, for given $k\in \Z$, the maximal dyadic cubes $\mD^k\subset \mD$ such that 
$\tfrac 1{|Q|}\sum_{R\subset Q}\|g_R\|>2^k$. Then
$\sett{x\in \R^n}{C_\mY g(x)>2^k}=\bigcup_{Q\in \mD^k}Q$, and the cubes in $\mD^k$ are disjoint.
We obtain
$$
  |\sett{x}{C_\mY g(x)>2^k}|=\sum_{Q\in \mD^k} |Q| \le 2^{-k} \sum_{Q\in\mD^k}\sum_{R\subset Q} \|g_R\|.
$$
Now let $\hat f_R:= \tfrac 1{|R|}\int_R C_\mY g$. Note that $\hat f_R$ does not depend on $k$,
and that $\hat f_R>2^k$ for $R\subset Q\in \mD^k$.
We get $|\sett{x}{C_\mY g(x)>2^k}|\le 2^{-k} \sum_{R: \hat f_R>2^k}\|g_R\|$
and
$$
  \|C_\mY g\|_{p'}^{p'} \approx \sum_{k\in\Z} 2^{p'k} |\sett{x}{C_\mY g(x)>2^k}|
  \le \sum_{R\in\mD} \sum_{k: 2^k <\hat f_R} 2^{(p'-1)k}\|g_R\|\approx \sum_{R\in\mD} (\hat f_R)^{p'-1} \|g_R\|.
$$
Now construct $f=(f_R)_{R\in\mD}$, choosing $f_R\in \mX_R$ such that $\|f_R\|=(\hat f_R)^{p'-1}$ and
$(\hat f_R)^{p'-1} \|g_R\|\approx \scl{f_R}{g_R}$.
We get $N_\mX f(x)= \sup_{Q\ni x}(\hat f_Q)^{p'-1}= (M_\mD(C_\mY g)(x))^{p'-1}$.
Since $p(p'-1)=p'$, this gives
$$
  \|N_\mX f\|_p^p= \|M_\mD(C_\mY g)\|_{p'}^{p'}\lesssim \|C_\mY g\|_{p'}^{p'},
$$
and we conclude that
$$
  \sum_Q\scl{f_Q}{g_Q}\gtrsim \|C_\mY g\|_{p'}^{p'} \gtrsim \|C_\mY g\|_{p'} \|N_\mX f\|_p.
$$

(iii)
Next we prove the estimate of $\|N_\mX f\|_{p}$.
Consider first the case $1<p<\infty$.
Select, for given $k\in \Z$, the maximal dyadic cubes $\mD^k\subset \mD$ such that 
$\|f_Q\|>2^k$. Then
$\sett{x\in \R^n}{N_\mX f(x)>2^k}=\bigcup_{Q\in \mD^k}Q$, and the cubes in $\mD^k$ are disjoint.
Write $k_Q:= \max_{Q\in\mD^k} k\le \log_2 \|f_Q\|$.
We obtain
\begin{multline*}
  \|N_\mX f\|_p^p\approx \sum_{k\in\Z} 2^{kp} |\sett{x}{N_\mX f(x)>2^k}| 
  = \sum_{Q\in\mD} |Q| \sum_{k: Q\in \mD^k} 2^{kp} \\
  \approx \sum_{Q\in\mD} |Q| 2^{k_Qp} 
  = \sum_{Q\in\mD} 2^{k_Q} |Q|  2^{k_Q(p-1)}
  \approx \sum_{Q\in\mD} \|f_Q\|\left( |Q| \sum_{k: Q\in \mD^k} 2^{k(p-1)} \right).
\end{multline*}
Write $\hat g_Q:=  |Q| \sum_{k: Q\in \mD^k} 2^{k(p-1)}$ and construct 
$g=(g_Q)_{Q\in\mD}$, choosing $g_Q\in \mY_Q$ such that $\|g_Q\|=\hat g_Q$ and
$\|f_Q\|\|g_Q\|\approx \scl{f_Q}{g_Q}$.
Then
\begin{multline*}
  \frac 1{|Q|}\sum_{R\subset Q}\|g_R\|\lesssim \sum_{k\in\Z} 2^{k(p-1)}\frac 1{|Q|}\sum_{R\subset Q, R\in \mD^k} |R| \\
  =  \sum_{k\in\Z}2^{k(p-1)} \frac 1{|Q|} |\sett{x}{N_\mX f(x)>2^k}\cap Q| \\
  \approx \frac 1{|Q|}\int_Q (N_\mX f)^{p-1}\le \inf_Q M_\mD ((N_\mX f)^{p-1}),
\end{multline*}
and therefore $\|C_\mY g\|_{p'}^{p'}\lesssim \|(N_\mX f)^{p-1}\|_{p'}^{p'}= \|N_\mX f\|_p^p$,
since $p'(p-1)=p$.
We conclude that
$$
  \sum_{Q\in\mD}\scl{f_Q}{g_Q}\gtrsim \|N_\mX f\|_{p}^{p} \gtrsim \|C_\mY g\|_{p'} \|N_\mX f\|_p.
$$

(iii')
We finally prove the estimate of $\|N_\mX f\|_1$, i.e. the case $p=1$.
Let $\mD^0$ be the $2^n$ dyadic cubes with sidelength $2^M$ and one corner at the origin,
where $M$ is chosen large enough, using the monotone convergence theorem, so that
$\|N_\mX \tilde f\|_1\ge \tfrac 12 \|N_\mX f\|_1$, where $\tilde f_Q:=f_Q$ if 
$Q\subset Q_0$ for some $Q_0\in\mD^0$, and $\tilde f_Q:= 0$ otherwise.
Assuming the estimate proved for $\tilde f$, we have
$$
  \|N_\mX \tilde f\|_1\lesssim \sum_Q \scl{\tilde f_Q}{g_Q}/\|C_\mY g\|_\infty,
$$
where we may assume $g_Q=0$ unless $Q\subset Q_0$ for some $Q_0\in\mD^0$.
This yields $\|N_\mX f\|_1\le 2\|N_\mX \tilde f\|_1\lesssim \sum_Q \scl{\tilde f_Q}{g_Q}/\|C_\mY g\|_\infty
\lesssim \sum_Q\scl{f_Q}{g_Q}/\|C_\mY g\|_\infty$.
Thus, replacing $f$ by $\tilde f$, we may assume that $f_Q=0$ unless $Q\subset Q_0$ for some $Q_0\in\mD^0$.

Given $f$ contained by $\mD^0$ as above, we define recursively sets of disjoint dyadic cubes $\mD^j\subset \mD$, $j=1,2,3,\ldots$, as follows.
Having constructed $\mD^j$, let $Q\in \mD^j$.
Define $\mD^{j+1}_Q$ to be the set of maximal dyadic cubes $R\in\mD$ such that $R\subset Q$ and $\|f_R\|> 2 \|f_Q\|$. Then let $\mD^{j+1}:= \bigcup_{Q\in\mD^j}\mD^{j+1}_Q$.
Furthermore, let $\mD^f:= \bigcup_{j}\mD^j$ and 
$$
   E(Q):= Q\setminus \bigcup_{R\in\mD^{j+1}_Q} R, \qquad Q\in \mD^j.
$$

From the above construction, if $x\in Q_k\subset Q_{k-1}\subset\ldots \subset Q_0$, where $Q_j\in\mD^j$, then $\|f_{Q_k}\|>2^{k-1}\|f_{Q_1}\|$, $k=2,3,\ldots$, where $\|f_{Q_1}\|>0$.
Hence, if $N_\mX f(x)<\infty$, then there is a minimal $Q\ni x$, $Q\in\mD^f$. For this $Q$, we have 
$x\in E(Q)$ and $N_\mX f(x)\le 2 \|f_Q\|$. Thus
\begin{equation}  \label{eq:anestimateofn}
  N_\mX f\le 2 \sum_{Q\in \mD^f} \|f_Q\| \, 1_{E(Q)}\qquad \text{a.e.}
\end{equation}
so that $\|N_\mX f\|_1 \le 2\sum_{Q\in \mD^f} \|f_Q\| |Q|$.
Conversely, if $x\in Q_k\subset Q_{k-1}\subset\ldots \subset Q_0$, where $Q_j\in\mD^j$,  are all the selected dyadic cubes containing $x$, then $N_\mX f(x)\ge \|f_{Q_k}\|\ge 2 \|f_{Q_{k-1}}\|\ge \ldots \ge 2^k\|f_{Q_0}\|$.
Thus
$$
\sum_{Q\in \mD^f} \|f_Q\| |Q|  = \int_{\R^n} \sum_{Q\in \mD^f, Q\ni x}\|f_Q\| \le 
 \int N_\mX f  \sum_{j=0}^\infty 2^{-j} \le 2\|N_\mX f\|_1.
$$ 

Now let $c\in(0,1)$ be a constant, to be chosen below, and define
$$
  \mD^f_1:= \sett{Q\in \mD^f}{|E(Q)|>c|Q| }\qquad\text{and}\qquad \mD^f_2:= \mD^f\setminus \mD^f_1.
$$
From \eqref{eq:anestimateofn} we have 
$$
  \|N_\mX f\|_1 \le 2\sum_{Q\in \mD^f_1} \|f_Q\| |Q| + 2 c\sum_{Q\in \mD^f_2} \|f_Q\| |Q|
  \le 2\sum_{Q\in \mD^f_1} \|f_Q\| |Q| + 4c \|N_\mX f\|_1.
$$
Choose $c=1/8$ to obtain
$\|N_\mX f\|_1 \le 4\sum_{Q\in \mD^f_1} \|f_Q\| |Q|$.
Construct $g=(g_Q)_{Q\in\mD}$, choosing $g_Q\in\mY_Q$ such that $\|g_Q\|= |Q|$ and $\scl{f_Q}{g_Q}\approx \|f_Q\| |Q|$ if $Q\in \mD^f_1$, and $g_Q:= 0$ otherwise.
Then $\|N_\mX f\|_1 \lesssim\sum_{Q\in \mD^f_1} \scl{f_Q}{g_Q}$.
To estimate
$$
  \frac 1{|Q|} \sum_{R\subset Q}\|g_R\|=  \frac 1{|Q|} \sum_{R\subset Q,R\in\mD^f_1}|R|,
$$
note that if $R\in\mD^f_1\cap\mD^j$, then $\sum_{R'\in\mD^{j+1}_R}|R'|\le 7/8 |R|$.
Thus
$$
\frac 1{|Q|} \sum_{R\subset Q,R\in\mD^f_1}|R| \le \frac 1{|Q|}\sum_{j=0}^\infty (7/8)^j|Q|= 8.
$$
Thus $\|C_\mY g\|_\infty\le 8$.
This completes the proof of the theorem.
\end{proof}

Consider now a duality $\scl\mX\mY$ between two Banach spaces $\mX$ and $\mY$ as in Definition~\ref{defn:duality}. We define the linear map $L: \mX\to \mY^*$ sending $f\in\mX$ to the linear functional
$$
  \Lambda_f: \mY\to \R: g\mapsto \scl fg.
$$
The estimate $|\scl fg|\le C \|f\|_\mX\|g\|_\mY$ shows that $\|L\|_{\mX\to\mY^*}\le C$, whereas it follows from
the estimate $ \|f\|_\mX \le C \sup_{\|g\|_\mY=1}\scl fg$ shows that $L$ is injective with closed range $L(\mX)\subset \mY^*$. Thus the duality gives a topological, but not in general isometric, identification, through $L$, of $\mX$ with a closed subspace $L(\mX)$ of $\mY^*$.
The estimate $\|g\|_\mY \le C \sup_{\|f\|_\mX=1}\scl fg$ furthermore shows that this subspace is ``large'' in the sense that its pre-annihilator is
$$
  {}^{\perp}L(\mX):= \sett{g\in\mY}{\Lambda g=0\text{ for all } \Lambda\in L(\mX)}=\{0\}.
$$
In general we may have that $L(\mX)\subsetneqq \mY^*$, but if $\mY$ is reflexive, then necessarily $L(\mX)=\mY^*$.
Below we identify $\mX$ and $L(\mX)$, and thus write $\mX= \mY^*$ if $L(\mX)=\mY^*$.
We also note that the above also holds with the roles of $\mX$ and $\mY$ interchanged, giving an identification of $\mY$ with a closed subspace of $\mX^*$.

The following result describes when the duality in Theorem~\ref{thm:dyadic} gives the full dual spaces.

\begin{thm}   \label{thm:fulldyadicduality}
  With the above notation, consider the duality $\scl{\mX_p}{\mY_{p'}}$,
$$
  f,g\mapsto \sum_{Q\in\mD} \scl{f_Q}{g_Q}
$$
from Theorem~\ref{thm:dyadic}. We have 
$\mY_{p'}\subsetneqq \mX_p^*$ for any $1\le p<\infty$, as well as $\mX_1\subsetneqq \mY_\infty^*$.

If furthermore the duality $\scl{\mX_Q}{\mY_Q}$ is such that $\mX_Q=\mY_Q^*$ for all $Q\in\mD$, and if $1<p<\infty$, then $\mX_p=\mY_{p'}^*$.
\end{thm}

\begin{proof}
(i)
We first prove $\mX_1\subsetneqq \mY_\infty^*$.
Let $Q_1\supsetneqq Q_2 \supsetneqq Q_3\supsetneqq\ldots$ be dyadic cubes.
Define the functionals 
$\Lambda_j g:= \scl{f_{Q_j}}{g_{Q_j}}$ on $\mY_\infty$, where we have chosen $f_{Q_j}\in\mX_{Q_j}$
such that $\|f_{Q_j}\|= 1/|Q_j|$.
It is clear that $\|\Lambda_j\|_{\mY_\infty^*}\approx 1$.
Consider the sequence space $\ell_\infty(\Z_+)$ and use Hahn--Banach's theorem to construct
$\lim\in(\ell_\infty(\Z_+))^*$ such that 
$$
  \lim((x_n)_{n=1}^\infty)= \lim_{n\to \infty} x_n
$$
for all convergent sequences $(x_n)_{n=1}^\infty$.
Set $\Lambda g:= \lim((\Lambda_j g)_{j=1}^\infty)$.
It is straightforward to verify that $\Lambda\in \mY_\infty^*\setminus \mX_1$.

(ii)
We next prove $\mY_{p'}\subsetneqq \mX_p^*$ for $1\le p<\infty$.
Fix some cube $Q_0\in\mD$ with $\ell(Q)=1$.
Define functionals 
$$
  \Lambda_j f:= \sum_{R: R\subset Q_0, \ell(R)= 2^{-j}} \scl{f_R}{g_R}
$$
on $\mX_p$,
where $g_R\in \mY_R$ is chosen such that $\|g_R\|=|R|$.
Then
$$
  |\Lambda_j f|\lesssim  \sum_{R: R\subset Q_0, \ell(R)= 2^{-j}} \|f_R\| |R|
  \le \int_{Q_0} N_\mX f\le \|N_\mX f\|_p.
$$
Define $\Lambda f:= \lim((\Lambda_j f)_{j=1}^\infty)$.
It is straightforward to verify that $\Lambda\in \mX_p^*\setminus \mY_{p'}$.

(iii)
Finally we assume that $\mX_Q=\mY_Q^*$ and $1<p<\infty$, and aim to show that $\mX_p=\mY_{p'}^*$.
Let $\Lambda\in \mY_{p'}^*$, and let $Q\in\mD$.
Pick $f_Q\in \mX_Q= \mY^*_Q$ such that $\scl{f_Q}{g_Q}= \Lambda((g_Q\delta_{QR})_{R\in\mD})$
for all $g_Q\in \mY_Q$, where $\delta_{QR}=1$ if $R=Q$ and $0$ otherwise.
Let $f:= (f_Q)_{Q\in\mD}$.
Then 
\begin{equation}   \label{eq:representation}
   \Lambda g= \sum_{Q\in\mD} \scl{f_Q}{g_Q}
\end{equation}
holds whenever $g_Q\ne 0$ only for finitely many $Q$.
From the monotone convergence theorem is follows that $\|N_\mX f\|_p\lesssim \|\Lambda\|_{\mY^*_{p'}}$,
so that $f\in\mX_p$.
We now use Lemma~\ref{lem:dense} below to deduce that \eqref{eq:representation} holds for all $g\in \mY_{p'}$ by continuity.
\end{proof}

\begin{lem}   \label{lem:dense}
Assume that $1<p'<\infty$. Then the subspace of finitely non-zero sequences $g=(g_Q)_{Q\in\mD}$
is dense in $\mY_{p'}$.
\end{lem}

\begin{proof}
(i)
  Let $g\in \mY_{p'}$ and let $\epsilon>0$.
  Let $Q_1,\ldots, Q_{2^n}$ be the dyadic cubes with one corner at the origin and sidelength $2^M$.
  Choose $M$ large enough so that 
  $\int_{\R_n\setminus Q_0} |C_\mY g|^{p'}\le \epsilon^{p'}$,
  where $Q_0:= Q_1\cup\ldots\cup Q_{2^n}$.
  Set 
$$
  g^1_Q:=
  \begin{cases}
    g_Q, & \qquad Q\not\subset Q_0,\\
    0, & \qquad Q\subset Q_0.
  \end{cases}
$$
Let $Q_j'$ be a sibling to $Q_j$, $1\le j\le 2^n$.
Since $g^1_Q=0$ for $Q\subset Q_j$, it is clear that
$\sup_{Q_j} C_\mY g\le \inf_{Q_j'} C_\mY g$.
Therefore
\begin{multline*}
  \|C_\mY g^1\|_{p'}^{p'}= \int_{\R^n\setminus Q_0} |C_\mY g^1|^{p'}+ \sum_{j=1}^{2^n} \int_{Q_j} |C_\mY g^1|^{p'} \\
  \le  \int_{\R^n\setminus Q_0} |C_\mY g|^{p'}+ \sum_{j=1}^{2^n} \int_{Q_j'} |C_\mY g|^{p'}
  \lesssim \epsilon^{p'},
\end{multline*}
since $C_\mY g^1\le C_\mY g$.

(ii)
Next we consider small cubes inside $Q_0$.
Define
$$
  C_j h(x) := \sup_{Q: x\in Q,\ell(Q)\le 2^{-j}} \frac 1{|Q|}\sum_{R\subset Q} \|h_R\|,\qquad h\in \mY_{p'}.
$$
Then $C_j g(x)\to 0$ as $j\to \infty$ for almost all $x$, by Lemma~\ref{lem:zerotrace} below.
Since $C_j g\le C_\mY g\in L_{p'}(\R^n)$, it follows by dominated convergence that 
we can choose $j<\infty$ such that $\|C_j g\|_{p'}\le \epsilon$.
Next choose $\delta>0$ such that 
$\sum_{R: R\subset Q_0, \ell(R)\le \delta} \|g_R\|\le \epsilon 2^{-nj}|Q_0|^{-1/p'}$.
  Set 
$$
  g^2_Q:=
  \begin{cases}
    g_Q, & \qquad Q\subset Q_0,\ell(Q)\le \delta,\\
    0, & \qquad \text{otherwise}.
  \end{cases}
$$
We have
\begin{multline*}
  C_\mY g^2(x)= \max\left( C_j g^2(x), \sup_{Q: x\in Q, \ell(Q)>2^{-j}}\frac 1{|Q|}\sum_{R\subset Q}\|g^2_R\|\right) \\
  \le \max\Big( C_j g^2(x),  \min(2^{nj}, d(x, Q_0)^{-n})\,  \epsilon 2^{-nj} |Q_0|^{-1/p'} \Big) \\
  \le \max\Big( C_j g^2(x), \epsilon |Q_0|^{-1/p'} \min(1, d(x, Q_0)^{-n})  \Big).
\end{multline*}
where $d(x,Q_0):= \inf_{y\in Q_0}|x-y|$.
This shows that 
$\|C_\mY g^2\|_{p'}\lesssim \epsilon$,
since
$$
  \| \min(1, d(x, Q_0)^{-n}) \|_{p'}\lesssim  |Q_0|^{1/p'}.
$$
It follows that $g-g^1-g^2$ is finitely non-zero, with $\|g^1+g^2\|_{\mY_{p'}}\lesssim \epsilon$.
\end{proof}

\begin{lem}     \label{lem:zerotrace}
   Let $Q_0\in\mD$ and assume that $\sum_{Q\subset Q_0} a_Q<\infty$, where $0\le a_Q<\infty$ for $Q\subset Q_0$.
   Then 
 $$
   \frac 1{|Q|}\sum_{R\subset Q} a_R\to 0,\qquad\text{as } Q\ni x, \ell(Q)\to 0,
 $$
 for almost all $x\in Q_0$.
\end{lem}

\begin{proof}
We argue by contradiction. Assume there exists $\delta>0$ such that
$$
  E:= \bigsett{x\in Q_0}{\limsup_{Q\ni x,\ell(Q)\to 0}  \frac 1{|Q|}\sum_{R\subset Q} a_R>\delta}
$$
has positive measure.
Let 
$A:= \sum_{Q\subset Q_0} a_Q<\infty$.
Choose $j<\infty$ such that $ \sum_{Q\subset Q_0, \ell(Q)> 2^{-j}} a_Q> A-\delta |E|/2$.
Select the maximal cubes $Q_k\subset Q_0$, $k=1,2,\ldots$, such that $\ell(Q_k)\le 2^{-j}$ and  
$\sum_{R\subset Q_k} a_R>\delta|Q_k|$.
We have that $E\subset \bigcup_k Q_k$, where the cubes $Q_k$ are disjoint.
This gives
\begin{multline*}
  A= \sum_{Q\subset Q_0} a_Q= \sum_{Q\subset Q_0,\ell(Q)>2^{-j}} a_Q + \sum_{Q\subset Q_0, \ell(Q)\le 2^{-j}} a_Q \\
  \ge (A- \delta |E|/2)+ \sum_k \delta |Q_k|\ge A+ \delta |E| /2,
\end{multline*}
which is a contradiction. The conclusion follows.
\end{proof}

\section{The non-dyadic results}    \label{sec:nondyadic}

In this section, we derive the corresponding non-dyadic results on the Carleson duality from the dyadic results in Section~\ref{sec:dyadic}.
We use the following notation.
For fixed constants $c_0>1, c_1>0, a>0$, we use Whitney regions
$W(t,x)$, $L_q$ Whitney averages $W_q f$  of functions $f\in L_q^\loc(\R^{1+n}_+)$,
and non-tangential maximal functions $N_* f$, as in the introduction.
Also define the Carleson functionals
$$
  C^r g(z):= \sup_{Q\ni z} \left(  \frac 1{|Q|}\iint_{\widehat Q} |g(t,x)|^r dtdx\right)^{1/r},\qquad z\in\R^n,
$$
for $1\le r<\infty$,
and the Hardy--Littlewood maximal function
$$
  M h(z):= \sup_{Q\ni z} \frac 1{|Q|}\int_Q h(y) dy, \qquad z\in\R^n,
$$
for $h\in L_1^{\text{loc}}(\R^n)$.
Here the suprema are over all (non-dyadic) axis-parallel cubes in $\R^n$ containing $z$.
We write $C^1g= Cg$ when $r=1$.

We aim to prove the following non-dyadic version of Corollary~\ref{cor:dyadicwithr}.

\begin{thm}   \label{thm:nondyadicduality}
Let $1/p+1/\tilde p=1/q+1/\tilde q=1/r$, with $r\le p<\infty$, $r\le q\le \infty$, $1\le r<\infty$.
Then there is a constant $0<C<\infty$ such that 
\begin{gather*}
  \|fg\|_{L_r(\R^{1+n}_+)} \le C \|N_*(W_q f)\|_{L_p(\R^n)} \|C^r(W_{\tilde q} g)\|_{L_{\tilde p}(\R^n)}, \\
 \|N_*(W_q f)\|_{L_p(\R^n)} \le C \sup_{\|C^r(W_{\tilde q} g)\|_{L_{\tilde p}(\R^n)}=1} \|fg\|_{L_r(\R^{1+n}_+)}, \\
 \|C^r(W_{\tilde q} g)\|_{L_{\tilde p}(\R^n)} \le C \sup_{ \|N_*(W_{q} f)\|_{L_p(\R^n)} =1}  \|fg\|_{L_r(\R^{1+n}_+)}.
\end{gather*}
\end{thm}

For $r=1$, this means that there is a duality
$$
  f,g \mapsto \iint_{\R^{1+n}_+} fg
$$
between the Banach spaces $N_{p,q}$ and $C_{p', q'}$, defined by the norms
$$
  \|f\|_{N_{p,q}}:= \|N_*(W_q f)\|_{L_p(\R^n)}
$$
and 
$$
  \|g\|_{C_{p',q'}}:= \|C^1(W_{q'} g)\|_{L_{p'}(\R^n)},
$$
with $1/p+1/p'=1$, $1/q+1/q'=1$, $1\le p<\infty$ and $1\le q\le \infty$.
We also prove the following non-dyadic version of Theorem~\ref{thm:fulldyadicduality}.

\begin{thm}    \label{thm:nondyadicdualspaces}
  With the above notation, consider the duality $\scl{N_{p,q}}{C_{p', q'}}$. 
  We have, for $1\le q\le \infty$,
$C_{p',q'}\subsetneqq (N_{p,q})^*$ for any $1\le p<\infty$, as well as $N_{1,q}\subsetneqq (C_{\infty,q'})^*$.

If $1<q\le\infty$ and $1<p<\infty$, then $N_{p,q}=(C_{p',q'})^*$.
\end{thm}

\begin{rem}[Relation to the Coifman--Meyer--Stein tent spaces]\label{rem:tentspaces}
It is immediate that for $\tilde{q}=r$, we have the pointwise equivalence $C^r(W_{\tilde{q}} g)=C^r(W_r g)\approx C^r g$. For $r=2$, this is the functional denoted simply by $C$ by Coifman, Meyer and Stein \cite{CMS}. They show \cite[Thm 3]{CMS} that there is further the $L^p$ equivalence
$$
  \|C^2(g)\|_{L_p(\R^n)}\approx\|A^2(g)\|_{L_p(\R^n)}=:\|g\|_{T_{p,2}},\qquad p\in(2,\infty),
$$
where
$$
  A^2(g):=\left(\iint_{|y-x|<t}|g(t,y)|^2\frac{dy\,dt}{t^n}\right)^{1/2}
$$
is the area integral and $T_{p,2}$ is the tent space. Observe also that $N_*(W_\infty g)$ is pointwise dominated by the non-tangential maximal function of $g$ with a different aperture, and hence
$$
  \|N_*(W_\infty g)\|_{L_p(\R^n)}\approx\|N_* g\|_{L_p(\R^n)}.
$$

In view of the previous observations, taking $\tilde{q}=r=2$ (and then $q=\infty$) in Theorem \ref{thm:nondyadicduality}, it gives the following characterization of pointwise multipliers from the tent space $T_{\tilde{p},2}$ to $L_2(\R^{1+n})$, where $1/p+1/\tilde{p}=1/2$ and $\tilde p>2$:
\begin{gather*}
  \|fg\|_{L_2(\R^{1+n}_+)} \le C \|N_*f\|_{L_p(\R^n)} \|g\|_{T_{\tilde p,2}}, \\
 \|N_* f\|_{L_p(\R^n)} \le C \sup_{\|g\|_{T_{\tilde p,2}}=1} \|fg\|_{L_2(\R^{1+n}_+)}, \\
 \|g\|_{T_{\tilde p,2}} \le C \sup_{ \|N_* f\|_{L_p(\R^n)} =1}  \|fg\|_{L_2(\R^{1+n}_+)}.
\end{gather*}

On the other hand, Theorem \ref{thm:nondyadicduality} does not contain the known duality results for these tent space, since duality in Theorem \ref{thm:nondyadicduality} corresponds to  $r=1$, and for this exponent the spaces appearing in the statement are outside the scale of classical tent spaces as introduced by Coifman, Meyer and Stein.
\end{rem}

We prove Theorems~\ref{thm:nondyadicduality} and \ref{thm:nondyadicdualspaces} by showing equivalence of the corresponding dyadic and non-dyadic norms. 
For this, we require the following two lemmata.

\begin{lem}   \label{lem:bigmean}
  Let $0\le u\in L_1^\loc(\R^{1+n}_+)$.
  Assume that $W\subset\bigcup_{j=1}^N W_j\subset \R^{1+n}_+$, where $|W_j|\le C|W|$ for $j=1,\ldots, N$.  
  Then for some $1\le j\le N$, we have
$$
  \frac 1{|W_j|}\iint_{W_j} u \ge \frac 1 {CN} \left( \frac 1{|W|}\iint_{W} u \right).
$$
\end{lem}

\begin{proof}
The conclusion follows directly from
$\iint_W u\le \sum_{j=1}^N \iint_{W_j} u\le N\max_{j}\iint_{W_j}u$.
\end{proof}

The following lemma uses the estimation technique from \cite[Lem. 1]{FS}.

\begin{lem}    \label{lem:estimatetechnique}
  Consider two functions $f,g:\R^n\to|0,\infty)$.
  Assume that there are constants $0<c_1,c_2<\infty$ such that
  $f(z)>\lambda$ implies $g>c_1\lambda$ on some set $B\subset \R^n$
  with $0<\sup\sett{|y-z|}{y\in B}^n\le c_2 |B|$. Then there is a constant $0<c_3<\infty$ such that
$$
   \|f\|_{L_p(\R^n)}\le c_3 \|g\|_{L_p(\R^n)},
$$
for any $1\le p\le \infty$.
\end{lem}

\begin{proof}
  Let $\lambda>0$.
  Let $E_\lambda:= \sett{y}{g(y)>c_1\lambda}$ and consider the indicator function $1_{E_\lambda}$.
  Let $z\in\R^n$ be such that $f(z)>\lambda$. Then, by hypothesis, there exists a set $B\subset E_\lambda$ and the hypothesis implies that
$$
   M(1_{E_\lambda})(z)\gtrsim |B|/\sup\sett{|y-z|}{y\in B}^n\ge c_2^{-1}>0.
$$
  By the weak $L_1$ boundedness of $M$, we have
$$
  |\sett{z}{f(z)>\lambda}|\le |\sett{z}{M(1_{E_\lambda})(z)\gtrsim 1}|\lesssim \|1_{E_\lambda}\|_1= |E_\lambda|.
$$
This proves the estimate for $p=\infty$. For $1\le p<\infty$, we estimate
\begin{multline*}
  \int_{\R^n} |f(x)|^p dx= \int_0^\infty  |\sett{z}{f(z)>\lambda}| p\lambda^{p-1}d\lambda \\
  \lesssim  \int_0^\infty  |\sett{z}{g(z)>c_1\lambda}| p\lambda^{p-1}d\lambda\approx \int_{\R^n} |g(x)|^p dx.
\end{multline*}
\end{proof}

In order to compare the Banach spaces $N_{p.q}$ and $C_{p', q'}$ with their dyadic counterparts,
we make the following definitions.
  With notation as in Section~\ref{sec:dyadic},
  denote by $N^\mD_{p,q}$ the space $\mX_p$ with $\mX_Q= L_q(W_Q)$, so that
$$
\|f\|_{N^\mD_{p,q}}= \|N_{L_q}(f)\|_{L_p(\R^n)}.
$$
  Similarly denote by $C^\mD_{p',q'}$ the space $\mY_{p'}$ with $\mY_Q= L_{q'}(W_Q)$, so that
$$
\|g\|_{C^\mD_{p',q'}}= \|C_{L_{q'}}(g)\|_{L_{p'}(\R^n)}.
$$
  In what follows, we shall identify functions $f\in L_1^\loc(\R^{1+n}_+)$
  and sequences $(f_Q)_{Q\in\mD}$ where $f_Q\in L_1(W_Q)$ in the natural way,
  i.e. given $f$ we set $f_Q:= f|_{W_Q}$ and given $(f_Q)_{Q\in\mD}$ we set
  $f:= f_Q$ on $W_Q$.
  
\begin{prop}   \label{prop:NTequivalentnorms}
  Let $1\le p<\infty$ and $1\le q\le \infty$.
  Under the above identification, the spaces $N_{p,q}$ and $N^\mD_{p,q}$ are equal,
  with equivalent norms
$$
   \|N_*(W_q f)\|_{L_p(\R^n)} \approx \|N_{L_q}(f)\|_{L_p(\R^n)}.
$$
In particular, up to equivalence of norms, the left hand side is independent of the exact choice of $a\ge 0, c_0>1, c_1>0$, and the right hand side is independent of the exact choice of dyadic system.
\end{prop}

Note that this shows that we here in fact can choose $a=0$, i.e. the vertical maximal function,
for $N_* (W_q f)$. This is because we already have some non-tangential control in $W_q f$.

\begin{proof}
(i)
 To prove the estimate $\|N_*(W_q f)\|_{L_p(\R^n)} \gtrsim \|N_{L_q}(f)\|_{L_p(\R^n)}$,
we use Lemma \ref{lem:estimatetechnique}. 
Assume $N_{L_q} f(z)>\lambda$.
Then there is a cube $Q\in\mD$ such that $z\in Q$ and 
$$
{|W_Q|}^{-1/q} \|f\|_{L_q(W_Q)}\ge \lambda.
$$
Consider (non-dyadic) cubes $W\subset \R^{1+n}_+$ with $\diam(W)= c_2\,  \dist(W, \R^n)$.
We fix $c_2>0$ small enough, depending on $c_0,c_1$, so that
$$
  W \subset \bigcap_{(s,y)\in W} W(s,y).
$$
 It is clear that there is an integer $N<\infty$ such that $W_Q$ is the union of at most $N$ such cubes $W$, uniformly for all $Q$.
 Lemma~\ref{lem:bigmean} shows that one of these cubes $W$, say $W_0$,
has ${|W_0|}^{-1/q} \|f\|_{L_q(W_0)}\gtrsim \lambda$.
It follows that
${|W(t,x)|}^{-1/q} \|f\|_{L_q(W(t,x))}\gtrsim \lambda$ for $(t,x)\in W_0$, and therefore 
$N_*(W_q f)\gtrsim \lambda$ on the projection $B\subset \R^n$ of $W\subset \R^{1+n}_+$,
and the stated estimate follows from Lemma~\ref{lem:estimatetechnique}.

(ii)
 Conversely, to prove the estimate $\|N_*(W_q f)\|_{L_p(\R^n)} \lesssim \|N_{L_q}(f)\|_{L_p(\R^n)}$,
 we again apply Lemma~\ref{lem:estimatetechnique}.
 Assume $N_*(W_q f)(z)>\lambda$.
 Then ${|W(t,x)|}^{-1/q} \|f\|_{L_q(W(t,x))}\ge \lambda$ for some $(t,x)$ such that $|x-z|\le at$.
 We see that that there is an integer $N<\infty$ such that $W(t,x)$ is contained in the union of at most $N$
  dyadic Whitney regions $W_Q$, with $N$ independent of $(t,x)$.
Thus by Lemma~\ref{lem:bigmean}, for some constant $c>0$,
${|W_Q|}^{-1/q} \|f\|_{L_q(W_Q)}\ge c \lambda$ for one of these $Q$.
Since $N_{L_q}(f) >c\lambda$ on $Q$ and $\dist(z,Q)\lesssim t\approx \ell(Q)$, Lemma~\ref{lem:estimatetechnique} completes the proof.
\end{proof}

\begin{prop}    \label{prop:Cequivalentnorms}
  Let $1\le p<\infty$ and $1\le q\le \infty$.
  Under the above identification, the spaces $C_{p',q'}$ and $C^\mD_{p',q'}$ are equal,
  with equivalent norms
$$
   \|C(W_{q'} g)\|_{L_{p'}(\R^n)} \approx \|C_{L_{q'}}g\|_{L_{p'}(\R^n)}.
$$
In particular, up to equivalence of norms, the left hand side is independent of the exact choice of $c_0>1, c_1>0$, and the right hand side is independent of the exact choice of dyadic system.
\end{prop}

\begin{proof}
  It is straightforward to check that the estimates below go through for $q'=\infty$ by properly interpreting the integrals.

  (i)
  To prove the estimate
  $$
  \|C(W_{q'} g)\|_{L_{p'}(\R^n)} \gtrsim \|C_{L_{q'}}g\|_{L_{p'}(\R^n)},
  $$
  assume that $C_{L_{q'}}g(z)>\lambda$.
  Then there is a cube $Q\in \mD$ such that $z\in Q$ and
  $$
  \frac 1{|Q|}\sum_{R\subset Q} |W_R|^{1-1/q'} \|g\|_{L_{q'}(W_R)}>\lambda.
  $$
  We claim that there is a constant $c>0$ such that
  $$
  \frac 1{|W_R|}\iint_{W_R} \left( \frac 1{|W(t,x)|} \iint_{W(t,x)} |g|^{q'} \right)^{1/q'} \ge c
  \left( \frac 1{|W_R|} \iint_{W_R} |g|^{q'} \right)^{1/q'}.
  $$
  Given this estimate, it follows that
  $$
    c\lambda< \frac 1{|Q|}\sum_{R\subset Q} \iint_{W_R} W_{q'} g=  \frac 1{|Q|}\iint_{\widehat Q} W_{q'}g\le C(W_{q'}g)(z),
  $$
  and hence $c\,C_{L_{q'}}g(z)\leq C(W_{q'}g)(z)$, even pointwise, from which the inequality in $L_{p'}(\R^n)$ follows.
    
  To prove the claimed reverse H\"older estimates, consider (non-dyadic) cubes $W\subset \R^{1+n}_+$
  with $\diam(W)= c_2\,  \dist(W, \R^n)$.
  We fix $c_2>0$ small enough, depending on $c_0,c_1$, so that
$$
  W \subset \bigcap_{(s,y)\in W} W(s,y).
$$
 It is clear that there is an integer $N<\infty$ such that $W_R$ is the union of at most $N$
 such cubes $W$, uniformly for all $R$.
 Lemma~\ref{lem:bigmean} shows that one of these cubes $W$, say $W_0$,
 has $\frac 1{|W_0|} \iint_{W_0} |g|^{q'} \gtrsim \frac 1{|W_R|} \iint_{W_R} |g|^{q'}$.
 We obtain
 \begin{multline*}
    \frac 1{|W_R|}\iint_{W_R} \left( \frac 1{|W(t,x)|} \iint_{W(t,x)} |g|^{q'} \right)^{1/q'} \\
    \gtrsim
     \frac 1{|W_0|}\iint_{W_0} \left( \frac 1{|W(t,x)|} \iint_{W(t,x)} |g|^{q'} \right)^{1/q'} 
     \gtrsim
      \frac 1{|W_0|}\iint_{W_0} \left( \frac 1{|W_0|} \iint_{W_0} |g|^{q'} \right)^{1/q'} \\
      =  \left( \frac 1{|W_0|} \iint_{W_0} |g|^{q'} \right)^{1/q'} 
      \gtrsim \left( \frac 1{|W_R|} \iint_{W_R} |g|^{q'} \right)^{1/q'}.
\end{multline*}

(ii)
  Conversely, to prove the estimate $\|C(W_{q'} g)\|_{L_{p'}(\R^n)} \lesssim \|C_{L_{q'}}g\|_{L_{p'}(\R^n)}$,
  assume that $C(W_{q'}g)(z)>\lambda$.
  Then there is a cube $Q$ such that $z\in Q$ and
  $$
  \frac 1{|Q|}\iint_{\widehat Q} W_{q'}g >\lambda.
  $$
  There is an integer $N<\infty$ such that $\bigcup_{(t,x)\in \widehat Q} W(t,x)\subset \bigcup_{j=1}^N \widehat Q_j=: U$ for some dyadic cubes $Q_j\in \mD$ with $\ell(Q)\le \ell(Q_j)\le N \ell(Q)$, $1\le j\le N$.  
Note that we can choose $N$ independent of $Q$.
Let $h:= |g| 1_U$ and $\widetilde W_R:= \bigcup_{(t,x)\in W_R} W(t,x)$,
and note that there are finitely many $S\in \mD$ such that $W_S$ intersect $\widetilde W_R$ (all with $\ell(S)\approx \ell(R)$), uniformly in $R$.
Then
\begin{multline*}
  \lambda |Q| <\iint_{\widehat Q}  \left( \frac 1{|W(t,x)|} \iint_{W(t,x)} h^{q'} \right)^{1/q'} 
  = \sum_{R\in \mD} \iint_{W_R}  \left( \frac 1{|W(t,x)|} \iint_{W(t,x)} h^{q'} \right)^{1/q'} \\
   \le \sum_{R\in \mD} |W_R|^{1-1/q'}\left( \iint_{W_R} \Big( \frac 1{|W(t,x)|} \iint_{W(t,x)} h^{q'} \Big) \right)^{1/q'} \\
   \lesssim\sum_{R\in \mD} |W_R|^{1-1/q'}\left( \iint_{\widetilde W_R}  h^{q'} \right)^{1/q'}
   \lesssim \sum_{R\in \mD} \sum_{S\in \mD: W_S\cap \tilde W_R\ne \emptyset} |W_S|^{1-1/q'}\left( \iint_{W_S}  h^{q'} \right)^{1/q'} \\
   =\sum_{S\in \mD} |W_S|^{1-1/q'}\left( \iint_{W_S}  h^{q'}  \right)^{1/q'}  \sum_{R\in \mD: \tilde W_R\cap W_S\ne \emptyset} 1\\
   \lesssim \sum_{S\in\mD: W_S\subset U}  |W_S|^{1-1/q'}  \|g\|_{L_{q'}(W_S)}
   \le \sum_{j=1}^N |Q_j| \inf_{Q_j} C_{L_{q'}} g.
\end{multline*}
Thus there is $c>0$ and $1\le j\le N$ such that $C_{L_{q'}}g>c\lambda$ on $Q_j$.
Lemma~\ref{lem:estimatetechnique} applies since we may assume $\dist(z,Q_j)\lesssim \ell(Q)\le \ell(Q_j)$.
\end{proof}

\begin{proof}[Proof of Theorem~\ref{thm:nondyadicduality}]
  The result follows from Corollary~\ref{cor:dyadicwithr}
  and Propositions~\ref{prop:NTequivalentnorms} and \ref{prop:Cequivalentnorms}.
  Note that by replacing $|f|^r$, $|g|^r$ by $f, g$, it suffices to consider the case $r=1$.
\end{proof}

\begin{proof}[Proof of Theorem~\ref{thm:nondyadicdualspaces}]
  The result follows from Theorem~\ref{thm:fulldyadicduality}
  and Propositions~\ref{prop:NTequivalentnorms} and \ref{prop:Cequivalentnorms}.
\end{proof}

\bibliographystyle{acm}
%GATHER{AKMcDirac.bib}  % makes sure WinEdt finds citations...
%\bibliography{carleson}

\end{document}